\documentclass[11pt,a4paper]{article}
\usepackage[T1]{fontenc}
\usepackage{lmodern}
\usepackage{amsmath,amssymb,amsthm}
\usepackage[margin=30mm]{geometry}
\usepackage{microtype}
\usepackage[hidelinks]{hyperref}

\newtheorem{theorem}{Theorem}[section]
\newtheorem{lemma}[theorem]{Lemma}
\newtheorem{proposition}[theorem]{Proposition}
\newtheorem{corollary}[theorem]{Corollary}

\theoremstyle{definition}
\newtheorem{definition}[theorem]{Definition}
\theoremstyle{remark}
\newtheorem{remark}[theorem]{Remark}

\newcommand{\R}{{2^\omega}}
\newcommand{\CH}{\mathsf{CH}}
\newcommand{\ZFC}{\mathsf{ZFC}}
\newcommand{\Sb}{\boldsymbol{\Sigma}}
\newcommand{\Pb}{\boldsymbol{\Pi}}
\newcommand{\Db}{\boldsymbol{\Delta}}
\newcommand{\Bb}{\mathbb B}
\newcommand{\Cc}{\mathbb C}
\newcommand{\Aa}{\mathbb A}
\newcommand{\Sel}{\operatorname{Sel}}
\newcommand{\WO}{\mathsf{WO}}
\newcommand{\SEP}{\mathsf{SEP}}
\newcommand{\Poss}{\operatorname{Poss}}
\newcommand{\Add}{\operatorname{Add}}

\title{On a local variant of the 12th Delfino problem\\
The \texorpdfstring{$\Pi$}{Pi}-side}
\author{Stefan Hoffelner\thanks{The author's research was funded in
whole by the Austrian Science Fund (FWF) Grant-DOI 10.55776/P37228.
For the purpose of open access, the author has applied a CC BY public
copyright license to any Author Accepted Manuscript version arising
from this submission.}\\ TU Wien}
\date{24 November 2026}

\begin{document}
\maketitle

\begin{abstract}
Assume that \(M_n\), the canonical inner model with \(n\) Woodin
cardinals, exists, where \(1\leq n<\omega\). We force a
cardinal-preserving extension of \(M_n\) in which \(\CH\) holds and
the reals have a lightface \(\Delta^1_{n+3}\) wellorder.
In this extension \(\Sb^1_{n+3}\)-separation holds and rules out
every good \(\Db^1_{n+3}\) wellorder of the reals. Every boldface
\(\Sb^1_{n+2}\) set of reals is Lebesgue
measurable and has the Baire property, and
\(\Pb^1_{n+1}\)-determinacy fails.

Thus the regularity and wellorder at the levels suggested by
\(M_{n+1}\) can already be obtained by forcing over \(M_n\).
For odd \(n\), \(\Db^1_{n+1}\)-determinacy also fails. Hence for
odd \(n\geq3\) we answer the adjacent-level Friedman--Schindler
question negatively, while for even \(n\geq2\) we answer its
\(\Pi\)-determinacy variant negatively.

We prove in \(\ZFC\) that for every \(k\geq2\), a \(\Db^1_k\)
wellorder of the reals precludes \(\Pb^1_k\)-reduction. Hence
\(\Pb^1_{n+3}\)-reduction and \(\Pb^1_{n+3}\)-uniformization fail in
our extension.
\end{abstract}

\section{Introduction}

Regularity and definable choice give two ways to describe the
structure of the projective sets of reals. Lebesgue measurability
and the Baire property are consequences of determinacy, while a
definable wellorder supplies a way to choose least witnesses. The
canonical mice with finitely many Woodin cardinals show how these
properties can occur at neighboring levels. If \(M_k\) is the
canonical inner model with \(k\) Woodin cardinals, Steel's analysis
gives a lightface \(\Delta^1_{k+2}\) wellorder of its reals
\cite{steelwellorder}. In \(M_{n+1}\), \(\Pb^1_{n+1}\)-determinacy
yields Lebesgue measurability and the Baire property
for \(\Sb^1_{n+2}\) sets, and the wellorder lies at
\(\Delta^1_{n+3}\) \cite[Section~1]{fs}. This raises the question of
how much of that configuration can be obtained over \(M_n\).

The global antecedent is Woodin's twelfth Victoria Delfino problem.
It asks whether projective determinacy follows if every projective
set of reals is Lebesgue measurable and has the Baire property, and
every projective relation admits a uniformization with a projective
graph. Woodin derived the existence of \(x^\dagger\) for every real
\(x\) and conjectured that the implication holds
\cite[Problem~12]{delfino}. Steel gave a
negative answer in 1997, and Schindler determined the consistency
strength of the regularity and uniformization hypotheses
\cite[Section~9]{schindlerlinear}.

Friedman and Schindler \cite[Section~1]{fs} asked a finite version
in which regularity is imposed at a specified projective level and
a lightface projective wellorder replaces the global uniformization
hypothesis. Their Theorem~1.2 gives a model over an inner model with
finitely many strong cardinals in which the \(\Sb^1_{n+2}\) sets are
universally Baire. It has a \(\Delta^1_{n+3}(a)\) wellorder for a
coding real \(a\) and a lightface \(\Delta^1_{n+5}\) wellorder.
They also record Friedman's earlier adjacent-level examples for
\(\Sb^1_3\) and \(\Sb^1_4\) regularity. Their strengthening asks
about the lightface wellorder at the adjacent level in general.
Questions about the levels of the regularity and of the wellorder
also occur in
Schindler's Bonn notes \cite[pp.~41--42 and 49]{schindlerbonn}.
In the indexing used here, we consider the following adjacent-level
form of their question.

\medskip
\noindent\textbf{Local question.}
Let \(2\leq n<\omega\). Suppose every \(\Sb^1_{n+2}\) set of reals
is Lebesgue measurable and has the Baire property and the reals
have a lightface \(\Delta^1_{n+3}\) wellorder. Does
\(\Db^1_{n+1}\)-determinacy follow?

\medskip
We also consider the variant obtained by replacing
\(\Db^1_{n+1}\)-determinacy by \(\Pb^1_{n+1}\)-determinacy.
The latter holds in the motivating model \(M_{n+1}\). We obtain
the regularity and wellorder hypotheses by forcing over \(M_n\),
together with separation at level \(n+3\). The resulting model
refutes the \(\Pi\)-determinacy variant for every \(n\), and the
\(\Delta\)-determinacy conclusion for odd \(n\).

We write \(\Sb^1_k\), \(\Pb^1_k\), and \(\Db^1_k\)
for the pointclasses allowing real parameters. A lightface
definition uses no real parameter. For a projective pointclass
\(\Gamma\), the separation property says that
\[
 A_0,A_1\in\Gamma,\quad A_0\cap A_1=\varnothing
 \quad\Longrightarrow\quad
 \exists X\in\Gamma\cap\check\Gamma\
 (A_0\subseteq X\subseteq\R\setminus A_1).
\]
\begin{theorem}\label{thm:main}
Suppose \(1\leq n<\omega\) and \(M_n\) exists. There is a
cardinal-preserving set-generic extension \(V\) of \(M_n\) with
\(V\models\CH\) in which
\begin{enumerate}
\item \(\Sb^1_{n+3}\)-separation holds,
\item \(\R\) admits a lightface \(\Delta^1_{n+3}\) wellorder and
admits no good \(\Delta^1_{n+3}(a)\) wellorder for any real \(a\),
\item every \(\Sb^1_{n+2}\) set is Lebesgue measurable and has the
Baire property,
\item \(\Pb^1_{n+1}\)-determinacy fails.
\end{enumerate}
If \(n\) is odd, then \(\Db^1_{n+1}\)-determinacy also fails.
\end{theorem}

Since \(\CH\) holds, the reals can be wellordered in type \(\omega_1\).
For a real \(a\), a \(\Delta^1_k(a)\) wellorder \(\prec\) of this
type is \emph{good} if the relation
\[
 \operatorname{IS}_\prec(s,y)
 \quad\longleftrightarrow\quad
 \{(s)_j:j\in\omega\}=\{z:z\prec y\}\cup\{y\}
\]
is \(\Delta^1_k(a)\), where \((s)_j\) is the \(j\)-th real in the
sequence coded by \(s\) \cite[Section~5A]{moschovakis}.
A \(\Sigma^1_k(a)\) definition of \(\operatorname{IS}_\prec\) suffices,
since an extra entry or a missing member of the initial segment
witnesses its \(\Sigma^1_k(a)\) complement.
Lemma~\ref{lem:good-wellorder} gives the failure of good wellorders
in Theorem~\ref{thm:main} from its separation clause.

For odd \(n\), the separation conclusion rules out
\(\Db^1_{n+1}\)-determinacy. Indeed, this determinacy hypothesis
implies \(\Sb^1_{n+3}\)-uniformization by Moschovakis's
Uniformization Theorem. This yields \(\Sb^1_{n+3}\)-reduction,
which is incompatible with separation for the same projective
pointclass \cite[Theorem~6C.5 and Exercise~4B.12]{moschovakis}.
Thus the local question has a negative answer for odd \(n\geq3\).
For even \(n\geq2\), the theorem answers its \(\Pi\)-determinacy
variant negatively. The failure of \(\Pb^1_{n+1}\)-determinacy for
all \(n\) is proved in Proposition~\ref{prop:nondeterminacy}.
The case \(n=1\) gives \(\Sb^1_3\) regularity, a lightface
\(\Delta^1_4\) wellorder and \(\Sb^1_4\)-separation.

The adjacent level is the first possible level for the wellorder.
If \(\prec\) is a \(\Db^1_k\) wellorder of \(2^\omega\), then the
\(\prec\)-least elements of the classes of \(E_0\), the relation of
eventual equality on \(2^\omega\), form a \(\Db^1_k\) set, since
these classes are countable. A transversal of \(E_0\) is not
Lebesgue measurable and does not have the Baire property.

This paper supersedes the earlier version of
\href{https://arxiv.org/abs/2606.13210}{arXiv:2606.13210}.
The \(\Pi^1_{n+3}\)-uniformization asserted there is incompatible
with the \(\Delta^1_{n+3}\) wellorder. Indeed, at every level
\(k\geq2\) a \(\Db^1_k\) wellorder rules out \(\Pb^1_k\)-reduction
(Theorem~\ref{thm:reduction-wellorder}), and hence
\(\Pb^1_k\)-uniformization
(Corollary~\ref{cor:wellorder-uniformization}). The forcing in
Theorem~\ref{thm:main} gives separation for \(\Sb^1_{n+3}\), the
level of the wellorder, instead.

\paragraph{Relation with the companion paper.}
The companion paper with Sandra M\"uller \cite{sigmaside} obtains
the same \(\Sb^1_{n+2}\) regularity and lightface
\(\Delta^1_{n+3}\) wellorder together with
\(\mathsf{MA}_{\omega_1}\) and \(\Sigma^1_{n+2+m}\)-uniformization
for every \(m<\omega\). Its construction uses coding over codes
and obtains regularity by applying Hjorth's argument \cite{hjorth}
in the model of Martin's axiom. Here the coding records
separation decisions and comparisons in the wellorder, and the
regularity comes from amoeba forcings for measure and category,
one pair for each real parameter.

In \(L\) at every level \(k\geq2\), and in the companion paper at
the levels \(n+2+m\), the \(\Sigma\)-classes have the uniformization
and reduction properties and the \(\Pi\)-classes have the separation
property. Here \(\Sb^1_{n+3}\) has the separation property, the
reverse of the pattern in \(L\). This is the sense of the subtitle.

\paragraph{The forcing construction.}
First, over \(M_n\) we add branches through an independent sequence
of Suslin trees and define a parameter-free \(\Sigma^1_{n+3}\)
coding predicate \(\Phi_n\). The trees, the coding forcings and
the decoding are those of \cite[Section~4]{piunif}. The coding is exact:
\(\Phi_n(u)\) holds precisely for the reals \(u\) coded at stages
of the iteration. To rule out an unintended code, we locate a
Suslin tree through which the code would give a branch and compare
the iteration with a forcing that leaves that tree Suslin.

Second, we use the coding iterations of
\cite[Subsection~2.5]{piunif}, which go back to
\cite[Definition~3.1]{reduction}. They consist of a countable-support
product of \(M_n\)-computed Cohen subsets of \(\omega_1\) followed by
a finite-support c.c.c.\ iteration. The Cohen subsets supply the
coding sets, and the factors of the iteration perform the coding and
add the regularity covers. Initial segments, restrictions and
products of these iterations are treated in
\cite[Lemma~2.15]{piunif}. We add amoeba and Cohen forcings as
further factors.

Third, a separation requirement tests which of two memberships
can be forced by a continuation from the class of allowable
forcings defined when the requirement is introduced. Products of
allowable forcings preserve the tests that determine the side of
the separator. For the simultaneous wellorder construction, the
orders at the boundaries of the actual iteration segments
end-extend one another. In the hypothetical continuations used
in the tests, the wellorders end-extend specified initial segments. This
permits a coherent system of wellorders on the intermediate
models of a product, as in Lemma~\ref{lem:amalgam}. We use
\cite{sep} and \cite[Lemmas~4.9 and 5.1]{sepred} for the separation
argument, and \cite{lsepwo} for the simultaneous construction
over \(L\). For distinct reals \(x,y\), exactly one of the two
order tags is coded. The two tags give complementary
\(\Sigma^1_{n+3}\) definitions on pairs of distinct reals.

Finally, a fixed tree \(T\) represents the universal
\(\Sb^1_{n+2}\) set in all the extensions used in the proof.
For each real \(a\), the amoeba forcings for measure and category over
\(L[T,a]\) add covers of the reals that fail to be random or Cohen over \(L[T,a]\),
respectively. On the complement of each cover, membership in the
corresponding section of \(p[T]\) is given by a Borel representative
of a Boolean value in \(L[T,a]\). This is the fixed-tree argument of
Friedman and Schindler \cite[Lemma~4.1]{fs}.

\paragraph{Organization.}
Section~\ref{sec:reduction-wellorder} proves the incompatibility in \(\ZFC\).
Section~\ref{sec:tree} gives the absolute tree representation.
Section~\ref{sec:coding} recalls the coding of \cite{piunif} and
introduces the amoeba forcings and the coding iterations.
Section~\ref{sec:allowable} proves
closure under products and invariance for allowable forcings.
Sections~\ref{sec:iteration} and~\ref{sec:verification} construct
the final iteration and verify the theorem.

Throughout, we identify the reals with \(2^\omega\) and code
finite tuples of reals and integers recursively. We use the
forcing order in which stronger conditions are smaller.

\section{Projective reduction and wellorders of the reals}
\label{sec:reduction-wellorder}

A \(\Gamma\)-reduction of a pair \(P_0,P_1\) consists of disjoint
\(J_i\in\Gamma\) with \(J_i\subseteq P_i\) for \(i<2\) and
\(J_0\cup J_1=P_0\cup P_1\).

\begin{theorem}\label{thm:reduction-wellorder}
Work in \(\ZFC\) and let \(k\geq2\). If \(\R\) admits a \(\Db^1_k\)
wellorder, then \(\Pb^1_k\)-reduction fails. The wellorder may have
any order type.
\end{theorem}

For reals \(v,p\) we write \(v\in\Delta^1_k(p)\) if
\(\{j\in\omega:v(j)=1\}\) is \(\Delta^1_k(p)\), and
\((\exists v\in\Delta^1_k(p))\) for quantification over such reals.
If \(\{v\}\in\Sigma^1_k(p)\), then \(v\in\Delta^1_k(p)\). Indeed, for
each \(i<2\), \(v(j)=i\) is equivalent to the existence of a real in
the singleton with \(j\)-th digit \(i\).

Fix a lightface \(\Pi^1_k\) predicate \(\mathcal V(d,p,j)\) universal for
\(\Pi^1_k\) formulas in the displayed arguments \(p,j\), with
formula index \(d\in\omega\). For \(e\) coding a pair of integers
\(((e)_0,(e)_1)\), put
\[
 P_i(e,p,j)\ \longleftrightarrow\ \mathcal V((e)_i,p,j) \qquad(i<2).
\]
We identify the triples \((e,p,j)\) with their real codes.
A \(\Pb^1_k\)-reduction of \(P_0,P_1\) gives disjoint
\(J_0,J_1\in\Pb^1_k\) with
\[
 P_i\setminus P_{1-i}\ \subseteq\ J_i \qquad(i<2).
\]
Suppose such \(J_i\) are given and fix one real \(q\) coding their
parameters.

\begin{lemma}\label{lem:restricted-quantification}
For each \(\Pi^1_k(q)\) predicate \(X(p,v)\), the predicate
\[
 p\longmapsto (\exists v\in\Delta^1_k(q,p))\,X(p,v)
\]
is \(\Pi^1_k(q)\), uniformly in a definition of \(X\).
\end{lemma}

\begin{proof}
Write \(r=\langle q,p\rangle\). Define
\begin{align*}
 \operatorname{Tot}(e,p) &\ \longleftrightarrow\
    (\forall j\in\omega)\bigl(J_0(e,r,j)\lor J_1(e,r,j)\bigr),\\
 \operatorname{Val}(e,p,v) &\ \longleftrightarrow\
    (\forall j\in\omega)\,\neg J_{1-v(j)}(e,r,j).
\end{align*}
Thus \(\operatorname{Tot}\) is \(\Pi^1_k(q)\) and \(\operatorname{Val}\)
is \(\Sigma^1_k(q)\). If \(\operatorname{Tot}(e,p)\) holds,
disjointness gives a unique \(v\) satisfying
\(\operatorname{Val}(e,p,v)\), with \(v(j)=i\) exactly when
\(J_i(e,r,j)\). The two digit predicates are complementary
\(\Pi^1_k(q,p)\) predicates, so \(v\in\Delta^1_k(q,p)\). Conversely,
given \(v\in\Delta^1_k(q,p)\), choose \((e)_i\) to define \(v(j)=i\)
with parameter \(r\). At every \(j\) the corresponding triple then
lies in \(P_{v(j)}\setminus P_{1-v(j)}\). Hence
\(\operatorname{Tot}(e,p)\) and \(\operatorname{Val}(e,p,v)\) hold.
It follows that
\begin{align*}
 &(\exists v\in\Delta^1_k(q,p))\,X(p,v)\\
 &\qquad\longleftrightarrow\
 (\exists e\in\omega)\bigl[\operatorname{Tot}(e,p)\land
     (\forall v)\bigl(\neg\operatorname{Val}(e,p,v)\lor X(p,v)\bigr)\bigr].
\end{align*}
The matrix of the universal real quantifier is \(\Pi^1_k(q)\),
and existential quantification over integers preserves this class.
\end{proof}

The lemma gives the restricted quantification used here directly
from the two sets \(J_i\). For the corresponding parametrization
and restricted quantification theorems for Spector pointclasses,
see \cite[4D.2--4D.3]{moschovakis}.

\begin{proof}[Proof of Theorem~\ref{thm:reduction-wellorder}]
Let \(\prec\) be the wellorder. Fix a real \(a\) coding \(q\) and
the parameters in both definitions of \(\prec\). For any
\(\Pi^1_{k-1}(a)\) predicate \(B(x,t)\), put
\[
 F(x,t)\ \longleftrightarrow\
 B(x,t)\ \land\
 \neg(\exists v\in\Delta^1_k(a,x,t))\bigl(v\prec t\land B(x,v)\bigr).
\]
The predicate inside the restricted quantifier is \(\Pi^1_k(a)\).
Since \(q\) is a coordinate of \(a\),
Lemma~\ref{lem:restricted-quantification} makes the restricted
existential assertion \(\Pi^1_k(a)\). As \(B\) also has a
\(\Sigma^1_k(a)\) definition, \(F\) is \(\Sigma^1_k(a)\).

Suppose \(B_x\ne\varnothing\), and let \(m\) be its \(\prec\)-least
element. Then \(m\in F_x\subseteq B_x\). If \(F_x\setminus\{m\}\)
is nonempty, let \(s\) be its \(\prec\)-least element. Now
\[
 \{m\}=\{t:t\prec s\land F(x,t)\}\in\Sigma^1_k(a,x,s).
\]
Thus \(m\in\Delta^1_k(a,x,s)\). Since \(m\prec s\) and \(B(x,m)\), the
real \(m\) witnesses the existential assertion excluded by \(F(x,s)\),
a contradiction. Hence \(F_x=\{m\}\), and now \(m\in\Delta^1_k(a,x)\).
In particular,
\[
 (\exists t)\,B(x,t)\ \longleftrightarrow\
 (\exists t\in\Delta^1_k(a,x))\,B(x,t).
\]
By the lemma the right side is \(\Pi^1_k(a)\). Choose \(B\) lightface
so that its projection \(S(u,z)\), with \(x=\langle u,z\rangle\),
is universal for the \(\Sb^1_k\) subsets of \(\R\). Then
\(\{z:\neg S(z,z)\}\in\Sb^1_k\), so it equals \(S_u\) for some
\(u\). Evaluation at \(u\) gives a contradiction.
\end{proof}

\begin{remark}
The proof uses only disjoint \(J_0,J_1\in\Pb^1_k\) with
\(P_i\setminus P_{1-i}\subseteq J_i\). The second separation
property for \(\Sb^1_k\) provides, for
all \(Y_0,Y_1\in\Sb^1_k\), disjoint \(J_0,J_1\in\Pb^1_k\) with
\(Y_i\setminus Y_{1-i}\subseteq J_i\). For
\(Y_i=\R\setminus P_{1-i}\) we have
\(Y_i\setminus Y_{1-i}=P_i\setminus P_{1-i}\), so this property
fails as well.
\end{remark}

\begin{corollary}\label{cor:wellorder-uniformization}
Let \(k\geq2\). In the presence of a \(\Db^1_k\) wellorder,
\(\Pb^1_k\)-uniformization fails. In the extension of
Theorem~\ref{thm:main}, \(\Pb^1_{n+3}\)-reduction and
\(\Pb^1_{n+3}\)-uniformization fail.
\end{corollary}

\begin{proof}
Let \(\bar0,\bar1\) be the constant reals with values \(0\) and \(1\),
and put
\[
 R=(P_0\times\{\bar0\})\cup(P_1\times\{\bar1\}).
\]
A \(\Pb^1_k\) uniformization \(G\) of \(R\) would reduce \(P_0,P_1\)
by \(J_0=\{z:(z,\bar0)\in G\}\) and \(J_1=\{z:(z,\bar1)\in G\}\),
contrary to Theorem~\ref{thm:reduction-wellorder}. For the last
assertion, take \(k=n+3\) and apply the second clause of
Theorem~\ref{thm:main}.
\end{proof}

\begin{lemma}\label{lem:good-wellorder}
Let \(k\geq2\). A good \(\Db^1_k\) wellorder of \(\R\) implies
\(\Sb^1_k\)-uniformization and \(\Sb^1_k\)-reduction. Hence
\(\Sb^1_k\)-separation precludes every such wellorder.
\end{lemma}

\begin{proof}
Fix a good \(\Delta^1_k(a)\) wellorder \(\prec\). We use Addison's
least-witness argument \cite[Exercise~5A.4]{moschovakis}.
Given a relation \(A\in\Sigma^1_k(b)\), choose
\(B\in\Pi^1_{k-1}(b)\) such that
\(A(x,y)\longleftrightarrow(\exists u)\,B(x,\langle y,u\rangle)\).
The \(\prec\)-least member of each nonempty section \(B_x\) is
specified by
\[
 H(x,t)\ \longleftrightarrow\
 B(x,t)\land(\exists s)\bigl[\operatorname{IS}_\prec(s,t)\land
   (\forall j\in\omega)\bigl((s)_j=t\lor\neg B(x,(s)_j)\bigr)\bigr].
\]
Thus \(H\in\Sigma^1_k(a,b)\), and
\(G(x,y)\longleftrightarrow(\exists u)\,H(x,\langle y,u\rangle)\)
uniformizes \(A\) in \(\Sigma^1_k(a,b)\).
For \(A_0,A_1\in\Sb^1_k\), uniformize
\((A_0\times\{\bar0\})\cup(A_1\times\{\bar1\})\), where
\(\bar0,\bar1\) are the constant reals. The two sections at
\(\bar0,\bar1\) of the uniformizing graph reduce \(A_0,A_1\).
Reduction and separation are incompatible for \(\Sb^1_k\)
\cite[Exercise~4B.12]{moschovakis}.
\end{proof}

\section{An absolute tree representation}\label{sec:tree}

Write $M=M_n$ and $\kappa=\omega_1^M$. A \emph{small
$M_n$-generic extension} is a generic extension of $M$ by a forcing
in $M$ which preserves $\kappa$ and has a dense subset of size at
most $\kappa$ in $M$. Such a forcing has the $\kappa^+$-chain
condition in $M$ and hence preserves all cardinals. Passing to the
dense subset gives the notion of \cite[Definition~4.3]{piunif}.

\begin{lemma}\label{lem:tree}
Fix a universal \(\Sb^1_{n+2}\) formula \(U\). There are
\(\lambda\in\mathrm{Ord}\) and a tree \(T\in M\) on
\(\omega\times\lambda\) such that in every small \(M_n\)-generic extension \(V\),
\[
 \forall e<\omega\ \forall a,x\in\R^V\
 \bigl(V\models U(e,a,x)\ \Longleftrightarrow\
                       [T_{e,a,x}]^V\ne\varnothing\bigr).
\]
Here \(T_{e,a,x}\) is the section at the recursive code of
\(\langle e,a,x\rangle\), and \([S]^V\) denotes the branches through
\(S\) in \(V\).
\end{lemma}

\begin{proof}
Apply \cite[Proposition~3.5]{mad}, which gives
\(p[T]^V=\{\langle e,a,x\rangle:V\models U(e,a,x)\}\).
Take the indicated sections of \(T\).
\end{proof}

\begin{lemma}\label{lem:absoluteness}
Let \(V\subseteq V'\) be small \(M_n\)-generic extensions and
\(\vec a\in(\R^V)^{<\omega}\). Then
\[
 V\models\varphi(\vec a)\ \Longleftrightarrow\ V'\models\varphi(\vec a)
 \qquad(\varphi\in\Sigma^1_{n+2}\cup\Pi^1_{n+2}).
\]
For \(\theta\in\Pi^1_{n+2}\),
\[
 V\models\exists z\,\theta(z,\vec a)
 \ \Longleftrightarrow\
 \exists z\in\R^V\ \bigl(V'\models\theta(z,\vec a)\bigr).
\]
In particular \(\Sigma^1_{n+3}\) truth with parameters in \(V\)
is upward absolute.
\end{lemma}

\begin{proof}
For \(S=T_{e,a,x}\in V\), a branch in \(V\) is a branch in \(V'\).
If \([S]^V=\varnothing\), a rank function in \(V\) remains a rank
function in \(V'\). Thus
\[
 V\models U(e,a,x)
 \ \Longleftrightarrow\ [S]^V\ne\varnothing
 \ \Longleftrightarrow\ [S]^{V'}\ne\varnothing
 \ \Longleftrightarrow\ V'\models U(e,a,x)
\]
by Lemma~\ref{lem:tree}. Take complements for \(\Pi^1_{n+2}\)
and apply the equivalence to \(\theta(z,\vec a)\) for \(z\in\R^V\).
\end{proof}

\section{The coding and the amoeba forcings}\label{sec:coding}

\subsection{The coding forcings}\label{sec:coding-forcings}

The coding forcings are those of \cite[Section~4]{piunif}. We recall
the objects and the facts used below and refer to \cite{piunif} for
the proofs.

\paragraph{Trees and branches.}
Let \(\vec S=\langle S_\xi:\xi<\kappa\rangle\) be the sequence
\(\vec S^{\,n}\) of \cite[Subsection~4.2]{piunif}. It is obtained by
the construction of \cite[Lemma~2.6]{piunif}, carried out in \(M\)
along the diamond sequence of \cite[Lemma~4.10]{piunif} with
\(<_M\)-least choices. Thus \(\vec S\) is an independent sequence of
normal Suslin trees in \(M\): every nonempty finite product of
distinct trees of \(\vec S\) is a Suslin tree in \(M\). The trees
split above every node. The nodes of \(S_\xi\) on level \(\lambda\)
carry ordinal labels \(\langle\lambda,i\rangle_{\rm ord}\), and we
regard a branch as the set of labels of its nodes. Adjoin a greatest
condition to each tree forcing when needed, and omit greatest
conditions from finite supports. Put
\[
 \Bb=\prod_{\xi<\kappa}^{\mathrm{fs}} S_\xi,
 \qquad W=M[H],
\]
where $H$ is $\Bb$-generic over $M$ and $b_\xi$ is the branch through
$S_\xi$ added by $H$. In \cite[Definition~4.11]{piunif} these are
\(\operatorname{Br}_n\) and \(W_n\). By \cite[Lemma~4.12,
Theorem~4.13 and Corollary~4.14]{piunif}, \(\Bb\) is c.c.c.\ of
size \(\kappa\), \(W\) has the cardinals of \(M\) and satisfies
\(\CH\), and for \(A\subseteq\kappa\) in \(M\) the product of the
trees \(S_\xi\), \(\xi\in A\), adds no branch through \(S_\xi\) for
\(\xi\notin A\).

\paragraph{The lower part.}
Let \(\mathcal I_n\) be the lightface \(\Pi^1_{n+1}\) set of real
codes of \cite[Definition~4.8]{piunif}. The premice it codes are
initial segments of \(M\) below \(\kappa\), their union is the lower
part \(m_\infty=M|\kappa\), and every small \(M_n\)-generic extension
has the same coded premice \cite[Lemma~4.9]{piunif}. Formally,
\(L[m_\infty,r]\) is built from a fixed set of ordinals in \(M\) which
codes \(m_\infty\) with all its predicates, and from which
\(m_\infty\) is recovered by collapsing. Similarly, \(L[T,a]\) is
built from a fixed set of ordinals in \(M\) coding \(T\).

\paragraph{Almost disjoint coding and localization.}
Let \(D=\langle d_\xi:\xi<\kappa\rangle\) be the almost disjoint
family of \cite[Subsection~4.4]{piunif}, defined in \(M\) from the
\(<_M\)-least enumeration of the reals of \(M\). For
\(Y\subseteq\kappa\), let \(\Aa_D(Y)\) be the almost disjoint coding
forcing for \(Y\) relative to \(D\), written \(\mathbb A^{\,n}(Y)\)
there. Its conditions are pairs
\((s,F)\in[\omega]^{<\omega}\times[Y]^{<\omega}\). Conditions with the
same first coordinate are centered, and the generic real \(r\) satisfies
\(Y=\{\xi<\kappa:|r\cap d_\xi|<\aleph_0\}\).

Let \(\Cc=({}^{<\kappa}2)^M\), ordered by end-extension. This is the
dense presentation by initial segments of \(\Add(\omega_1,1)^M\). For
a \(\Cc\)-generic \(g\), let \(h(g)\subseteq\kappa\) be its coding
set and put
\[
 \Sel(h,u)=\{\omega\gamma+2k+u(k):\gamma\in h,\ k<\omega\},
\]
as in \cite[Subsection~4.4]{piunif}. The set \(h(g)\) is unbounded in
\(\kappa\), and for distinct \(g,g'\) the intersection
\(h(g)\cap h(g')\) is bounded in \(\kappa\). For \(u\in\R\), let
\(Y(g,u)\subseteq\kappa\) be the set \(Y_n(u,g)\) of
\cite[Subsection~4.4]{piunif}. It codes \(u\), \(h(g)\) and the
branches \(b_\xi\) for \(\xi\in\Sel(h(g),u)\), and it is defined from
\(g\), \(u\), these branches and parameters in \(M\). To code \(u\),
a stage of the iterations in Subsection~\ref{subsec:mixed} adds a
Cohen subset \(g\) of \(\kappa\) at a new coordinate and then forces
with \(\Aa_D(Y(g,u))\).

\begin{lemma}[{\cite[Lemma~4.15]{piunif}}]\label{lem:decoder}
There are parameter-free formulas
\[
 \Psi_n(r,u)\in\Pi^1_{n+2},
 \qquad
 \Phi_n(u)\ \longleftrightarrow\ \exists r\,\Psi_n(r,u),
 \qquad
 \Phi_n\in\Sigma^1_{n+3},
\]
with the following properties.
\begin{enumerate}
\item If \(r\) is the real added at a coding stage for \(u\)
(Definition~\ref{def:mixed}), then \(\Psi_n(r,u)\) holds in the
resulting extension and in every further small \(M_n\)-generic
extension.
\item If \(\Psi_n(r,u)\) holds in a small \(M_n\)-generic extension,
then \(u\in L[m_\infty,r]\), and there is an unbounded
\(h\subseteq\kappa\) in \(L[m_\infty,r]\) such that
\(L[m_\infty,r]\) contains a cofinal branch through \(S_\xi\) for
every \(\xi\in\Sel(h,u)\).
\end{enumerate}
\end{lemma}

The decoding rests on the localization of \(Y(g,u)\) to
countable transitive models and on the recognition of \(m_\infty\)
through \(\mathcal I_n\) \cite[Subsection~4.4]{piunif}.

\subsection{The amoeba forcings}\label{subsec:amoeba}

Let $N_a=L[T,a]$, where $a\in\R$ belongs to the ambient model.
Write \(\mathcal N\) and \(\mathcal M\) for the ideals of Lebesgue
null and of meager subsets of \(2^\omega\).
We describe the two forcings as computed in $N_a$. When either is
used in a larger model, its conditions and its order are those
computed in $N_a$. The open and closed sets in conditions are given by
Borel codes in \(N_a\), and the covers added by the generic are defined
from the interpretations of these codes in the extension. A countable
family of linked sets
in $N_a$ covering the forcing still witnesses that it is
$\sigma$-linked.

\paragraph{Measure.}
Let $\mu$ be the product measure on $2^\omega$. For $k\geq1$ put
\[
 Q_k^{\mathcal N}=\{O:O\text{ is open and }\mu(O)<2^{-k}\},
 \qquad O'\leq O\ \Longleftrightarrow\ O\subseteq O' .
\]
This order is $\sigma$-linked. Given $O$, choose a clopen
$O_0\subseteq O$ and a positive rational $\varepsilon$ such that
\[
        \mu(O\setminus O_0)<\varepsilon,\qquad
        \mu(O_0)+2\varepsilon<2^{-k}.
\]
Such a pair exists by inner approximation of $O$ by clopen sets and
the strict bound on $\mu(O)$. Two conditions with the same pair are
compatible, with union as a common extension. Define
$\Aa^{\mathcal N}_{N_a}=\prod_{k\geq1}^{\mathrm{fs}}Q_k^{\mathcal N}$.
It is $\sigma$-linked, since the finite supports and the countably
many linked pieces at each coordinate give countably many linked sets
covering it.

For a generic filter write $O_k^*$ for the union of its open sets
at coordinate $k$. Directedness and a countable subcover give
$\mu(O_k^*)\leq 2^{-k}$. If $B$ is a null Borel set coded in $N_a$,
the conditions whose open set at coordinate $k$ contains $B$ are
dense. Indeed, for $O\in Q_k^{\mathcal N}$ choose an open $O'\supseteq B$
with $\mu(O')<2^{-k}-\mu(O)$ and use $O\cup O'$. Thus the Borel set
\[
                        Z_{\mathcal N}=\bigcap_{k\geq1}O_k^*
\]
is null and contains every null Borel set coded in $N_a$.

\paragraph{Category.}
Let $\tau_F=\{s\in2^{<\omega}:[s]\cap F\ne\varnothing\}$ for a closed
$F\subseteq2^\omega$. Let $Q^{\mathcal M}$ consist of pairs $(k,F)$
with $k<\omega$ and $F$ closed nowhere dense. The order is
\[
  (\ell,F')\leq(k,F)
  \ \Longleftrightarrow\
  k\leq\ell,\quad F\subseteq F',\quad
  \tau_{F'}\cap2^{\leq k}=\tau_F\cap2^{\leq k}.
\]
The conditions with the same $k$ and the same finite tree are
centered, using finite unions of their closed sets. Adjoin a greatest
condition to \(Q^{\mathcal M}\) for use at omitted coordinates. Set
$\Aa^{\mathcal M}_{N_a}=\prod_{j<\omega}^{\mathrm{fs}}Q^{\mathcal M}$.
This forcing is $\sigma$-centered. At each coordinate $j$ let
$F_j^*$ be the closure of the union of the closed sets in the
generic filter there. Given a string $s$ and a condition $(k,F)$,
choose $t\supseteq s$ with $[t]\cap F=\varnothing$ and increase
$k$ to at least $|t|$. This ensures that $[t]\cap F_j^*=\varnothing$.
The corresponding dense sets show that $F_j^*$ is nowhere dense.
For every closed nowhere dense $F$ coded in $N_a$, inserting $(0,F)$
at a fresh coordinate is dense. Since every Borel meager set coded
in $N_a$ is covered in $N_a$ by a countable union of closed nowhere dense
sets, the Borel set
\[
                        Z_{\mathcal M}=\bigcup_{j<\omega}F_j^*
\]
contains all such Borel meager sets.

In either case, a generic over a larger model for the forcing
computed in $N_a$ meets every dense subset of it in $N_a$. The descriptions above give the
covers already in the extension of $N_a$ by that generic. Borel
inclusion, nullity, and meagerness then persist to the ambient
extension.

\subsection{Coding iterations}\label{subsec:mixed}

We use the coding iterations of \cite[Definition~2.14 and
Subsection~4.4.1]{piunif}, with amoeba and Cohen forcings as
additional factors. For \(E\in M\) put
\[
 \Cc_E=\operatorname{Fn}_{<\kappa}(E\times\kappa,2)^M,
\]
ordered by reverse inclusion. On each coordinate, the conditions
defined on initial segments of \(\kappa\) give a dense copy of
\(\Cc\). Let \(E_i=[\kappa\cdot i,\kappa\cdot(i+1))\) and
\(E_{<\lambda}=\bigcup_{i<\lambda}E_i\).

\begin{definition}\label{def:mixed}
A \emph{coding iteration} of length \(\lambda\leq\kappa\) over \(W\)
is defined as in \cite[Definition~2.14]{piunif}. Stage \(i<\lambda\)
forces first with \(\Cc_{E_i}\) and then with a factor \(\dot\Aa_i\).
We write \(P_i\) for the initial segment through the first \(i\)
stages. The following data are given by \(P_i\)-names: the
\emph{type} of the stage, a real parameter, an index
\(\eta_i\in E_i\), and an isomorphism
\(\varphi_i:\Cc_{E_i}\cong\Cc_{E_i}\restriction c_i\) in \(M\) of the
form described there. If \(K_i\) is \(\Cc_{E_i}\)-generic, \(g_i\)
denotes the \(\eta_i\)-th Cohen subset of \(\kappa\) obtained from
the filter generated by \(\varphi_i[K_i]\). According to its type,
the stage is
\begin{enumerate}
\item a \emph{coding stage} for the real \(u_i\), with
\(\Aa_i=\Aa_D(Y(g_i,u_i))\),
\item an \emph{amoeba stage} for the real \(a\) and
\(\mathcal I\in\{\mathcal N,\mathcal M\}\), with
\(\Aa_i=\Aa^{\mathcal I}_{N_a}\),
\item a \emph{Cohen stage}, with \(\Aa_i=2^{<\omega}\) ordered by
end-extension, or
\item a \emph{trivial stage}, with \(\Aa_i\) trivial.
\end{enumerate}
The factor may be restricted below a condition \(t_i\) given by a
\(P_i*\Cc_{E_i}\)-name. The Cohen forcings form the product with
countable supports computed in \(M\), and the factors \(\dot\Aa_i\)
are iterated with finite support:
\[
 P_\lambda=\Cc_{E_{<\lambda}}*\dot{\mathbb S}_\lambda,
 \qquad
 \mathbb S_\lambda=\mathop{*}^{\mathrm{fs}}_{i<\lambda}\dot\Aa_i .
 \tag{1}\label{eq:coding-iteration}
\]
\end{definition}

As in \cite[Subsection~2.5]{piunif}, the names for the factors
\(\dot\Aa_i\) in \eqref{eq:coding-iteration} are translated from
\(P_i*\Cc_{E_i}\), and \(\mathbb S_\lambda\) has a dense suborder
consisting of finite functions \(f\) defined on stages. The value
\(f(i)\) denotes a condition of \(\dot\Aa_i\). At a coding stage it
is a pair in \([\omega]^{<\omega}\times[\kappa]^{<\omega}\), and at a
Cohen stage an element of \(2^{<\omega}\). At an amoeba stage it is
an ordinal \(\vartheta\), which denotes the \(\vartheta\)-th condition
of \(\Aa^{\mathcal I}_{N_a}\) in the canonical wellorder of \(N_a\).
The pairs \(p=(c,f)\), with \(c\in\Cc_{E_{<\lambda}}\) forcing \(f\)
into this suborder, form the \emph{dense suborder} of \(P_\lambda\).

A \emph{tail} over an extension \(V\) of \(W\) by a coding
iteration, possibly restricted below a condition, is the quotient of
a longer coding iteration by the generic for that initial segment.
As in \cite[Subsection~2.5]{piunif}, a tail has a fixed length and
uses Cohen index sets disjoint from those of the preceding
iteration. The bijections of index sets used to present products and
quotients belong to \(M\). The complete embeddings of initial
segments and the corresponding translations of names are part of
the presentation.

\begin{lemma}\label{lem:presentation}
The assertions of \cite[Lemmas~2.15 and~2.16]{piunif} hold for the
coding iterations of Definition~\ref{def:mixed}. Thus, for a coding
iteration \(P_\lambda\) with \(\lambda\leq\kappa\), we have the
following.
\begin{enumerate}
\item For \(\beta\leq\lambda\), \(P_\beta\) is a complete subforcing
of \(P_\lambda\). On the dense suborder the projection is
\[
 \pi_\beta(c,f)=
 \bigl(c\restriction(E_{<\beta}\times\kappa),f\restriction\beta\bigr).
\]
If \(G_\beta\subseteq P_\beta\) is generic, then
\(P_\lambda/G_\beta\) is the coding iteration of the stages in
\([\beta,\lambda)\) over \(W[G_\beta]\).
\item For \(p\) in the dense suborder, \(P_\lambda\restriction p\) is
again a coding iteration, with initial segments
\(P_\beta\restriction\pi_\beta(p)\), \(\beta\leq\lambda\). A name for
a tail of fixed length which is given below a Boolean value of the
preceding iteration lifts to a name for a tail of the same length
whose factors are trivial off that value.
\item Two tails \(P,Q\) over the same model have product
presentations in either order, on disjoint Cohen index sets. With
\(P\) first, each \(P\times Q_j\) is an initial segment of
\(P\times Q\). A name for a continuation of \(Q_j\) translates to a
name for a continuation of \(P\times Q_j\), preserving all specified
embeddings and translations of names.
\item \(P_\lambda\) preserves cardinals and \(\CH\) over \(W\). The
forcing \(\Bb*\dot P_\lambda\) has a dense subset of size at most
\(\kappa\) in \(M\), and the same holds for the forcings obtained
from it by the operations in (1)--(3). Hence all their generic
extensions are small \(M_n\)-generic extensions. If
\(G\subseteq P_\kappa\) is generic and \(G_\beta=G\cap P_\beta\),
then
\[
 \R^{W[G]}=\bigcup_{\beta<\kappa}\R^{W[G_\beta]} .
\]
\end{enumerate}
\end{lemma}

\begin{proof}
For iterations with coding and trivial stages only, this is
\cite[Lemmas~2.15 and~2.16]{piunif}, in the form for \(M_n\) given
in \cite[Subsection~4.4.1]{piunif}. The proofs there use three
properties of the factor \(\Aa_D(Y(g_i,u_i))\restriction t_i\).
\begin{itemize}
\item[(i)] Its conditions and its order are determined by the values
of the stage parameters and of \(t_i\).
\item[(ii)] Its conditions are denoted by elements of \(W\), and at
most \(\kappa\) of these occur at each stage.
\item[(iii)] It is c.c.c.
\end{itemize}
Property (i) allows the later Cohen forcings to be moved before the
earlier factors, which gives \eqref{eq:coding-iteration}. It also
gives the agreement of the dense suborders computed in different
models, which is used for restrictions, lifts, products and the
localization of reals. Property (ii) gives the dense suborder of
finite functions and the bound on its size, and (iii) gives the
chain condition of \(\mathbb S_\lambda\) after the Cohen product.

Cohen forcing has the three properties. At an amoeba stage for
\(a\), the forcing \(\Aa^{\mathcal I}_{N_a}\) and the wellorder used
for its indices are defined in \(N_a=L[T,a]\) from \(a\) and the
fixed code of \(T\) in \(M\). This gives (i). The forcing is
\(\sigma\)-linked by Subsection~\ref{subsec:amoeba}, which gives
(iii). For (ii), suppose inductively that the forcing
\(P_i*\Cc_{E_i}\) preceding the factor has a dense subset of size at
most \(\kappa\) and forces \(\CH\). Then it forces that
\(\Aa^{\mathcal I}_{N_a}\) has at most \(\kappa\) conditions, so that
their indices are below \(\kappa^+\). By the \(\kappa^+\)-chain
condition there is \(\theta_i<\kappa^+\) such that only indices
below \(\theta_i\) occur. With these observations the proofs in
\cite{piunif} apply to the present iterations.
\end{proof}

\begin{lemma}\label{lem:exact}
Let \(W[G]\) be a generic extension by a coding iteration of length
\(\lambda\leq\kappa\). For every \(u\in\R^{W[G]}\),
\[
 W[G]\models\Phi_n(u)
 \quad\Longleftrightarrow\quad
 u=u_i\text{ for some coding stage }i<\lambda .
\]
\end{lemma}

\begin{proof}
For iterations with coding and trivial stages only, this is
\cite[Theorem~4.16]{piunif}; see also \cite[Lemma~3.4]{reduction}.
If \(i\) is a coding stage, the real added by \(\dot\Aa_i\)
witnesses \(\Phi_n(u_i)\) in \(W[G]\), by
Lemma~\ref{lem:decoder}(1) and Lemma~\ref{lem:presentation}(4).

For the converse we follow the proof of
\cite[Theorem~4.16]{piunif}, which refers to that of
\cite[Lemma~2.17]{piunif}, and indicate the changes. Suppose
\(\Psi_n(r,u)\) holds in \(W[G]\) and no stage codes \(u\). Let
\(K\) be the generic for the Cohen product. The proof in
\cite{piunif} uses a countable set \(J\) of stages which contains
the supports of names for \(r\) and \(u\) and is closed under the
supports of the names for the stage parameters. It gives a countable
\(B\subseteq\kappa\) in \(M\) such that these names are coded in
\(M[H_B]\), where \(H_B\) is the restriction of \(H\) to the trees
with index in \(B\). For \(\xi\) outside \(B\) and outside the sets
\(\Sel(h(g_i),u_i)\) of the coding stages \(i\in J\), it forms a
truncated iteration of the stages in \(J\) over
\[
 M_\xi=M[K][H\restriction(\kappa\setminus\{\xi\})],
\]
whose generic extension \(M_\xi[\widetilde G]\) contains \(r\) and
\(u\). In the present situation the stage parameters include the
types, the real parameters of the amoeba stages and the conditions
\(t_i\). Their names are again coded by countable sets of ordinals,
since conditions are denoted by pairs of finite sets, finite
sequences and ordinals. At an amoeba or Cohen stage in \(J\) the
truncated iteration uses the same factor as \(P_\lambda\). For an
amoeba stage this factor is computed in \(L[T,a]\) from the value of
the name for \(a\), and \(T\in M\). No branches of the trees are used
at these stages. Thus the generic for the stages in \(J\) induces
\(\widetilde G\) as in \cite{piunif}.

The tree \(S_\xi\) is Suslin in \(M_\xi\) by
\cite[Lemma~2.5(2)]{piunif}. The factors of the truncated iteration
are \(\sigma\)-centered at coding stages, \(\sigma\)-linked at amoeba
stages and countable at Cohen stages. The product of such a forcing
with a c.c.c.\ forcing is c.c.c. Hence \(S_\xi\) remains Suslin in
\(M_\xi[\widetilde G]\), as it does for iterations of coding stages.

It remains to choose \(\xi\). By Lemma~\ref{lem:decoder}(2) there
are unboundedly many \(\gamma<\kappa\) such that \(L[m_\infty,r]\)
contains a cofinal branch through \(S_{\omega\gamma+2k+u(k)}\) for
every \(k<\omega\). Choose such a \(\gamma\) above \(B\) and above
the pairwise intersections of the sets \(h(g_i)\) for the coding
stages \(i\in J\). At most one of these sets contains \(\gamma\). If
there is one, choose \(k\) with \(u(k)\ne u_i(k)\), and put \(k=0\)
otherwise. Then \(\xi=\omega\gamma+2k+u(k)\) is as required. Since
the code of \(m_\infty\) belongs to \(M\), we have
\(L[m_\infty,r]\subseteq M_\xi[\widetilde G]\). This model therefore
contains a cofinal branch through the Suslin tree \(S_\xi\), a
contradiction.
\end{proof}

\section{Allowable forcings and end-extension requirements}\label{sec:allowable}

Fix an enumeration $\langle\varphi_e(x,a):e<\omega\rangle$ of the
$\Sigma^1_{n+3}$ formulas and write $A_e(a)=\{x:\varphi_e(x,a)\}$.
We use coded tuples of two forms,
\[
 \langle\WO,x,y\rangle,\qquad
 \langle\SEP,c,a,e_0,e_1,x,i\rangle\quad(i<2).
\]
They have disjoint recursive ranges. Every coding stage at which a
tuple of one of these forms is coded is subject to the corresponding
rule below.

An \emph{ordered iteration} over a model $V$ is a tail over $V$
(Subsection~\ref{subsec:mixed}) together with, for every initial segment and for the whole iteration,
a name for a wellorder of the reals of the corresponding extension.
The order at stage zero is the given order of the reals of \(V\).
Each later wellorder extends the earlier ones on their reals.
An \emph{end-extension requirement} is an initial segment \(R_0\)
of the wellorder of \(V\), given as a set in \(V\). It requires
every later wellorder to end-extend the wellorder of \(R_0\): the
set \(R_0\) remains an initial segment, with the same order. Reals of
\(V\) outside such segments carry no requirement, and a new real may
be placed below them.

A \emph{list of requirements} is a sequence $\mathcal L$ of length less
than \(\kappa\) whose entries are end-extension requirements and
separation requirements
\[
 b=(c,a,e_0,e_1,\sigma).
\]
We write \(\mathcal L_\eta\) for the sequence of its first \(\eta\)
entries. The ordinal $\sigma$ refers to the class of allowable forcings defined
from the list just before $b$ was added. In particular a requirement
added after the list $\mathcal L_\sigma$ refers only to $\mathcal L_\sigma$ or to an
earlier list. Write $\mathcal C_\eta(V)$ for the class of
\emph{$\eta$-allowable} forcings over $V$, defined from
\(\mathcal L_\eta\) and the given wellorder of $V$. Its elements are
ordered iterations over \(V\), and their restrictions below conditions
of the dense suborder, which satisfy the end-extension requirements
in \(\mathcal L_\eta\). At every stage of such a forcing, including the
forcings considered in the tests below, we require the following.
\begin{enumerate}
\item A tuple $\langle\WO,x,y\rangle$ is coded only if $x\ne y$ and
$x$ precedes $y$ in the wellorder at that stage.
\item Let $b=(c,a,e_0,e_1,\sigma)\in\mathcal L_\eta$. A tuple
$\langle\SEP,c,a,e_0,e_1,x,i\rangle$ is coded only for the following
value of $i$: $i=0$ if $\Poss_\sigma(\varphi_{e_0}(x,a))$ holds at that
stage; $i=1$ if this fails and $\Poss_\sigma(\varphi_{e_1}(x,a))$ holds;
and $i=0$ if both fail.
\end{enumerate}
Here, for an intermediate model $V'$, we write
\[
 \Poss_\sigma^{V'}(\theta)
 \quad\Longleftrightarrow\quad
 \exists Q\in\mathcal C_\sigma(V')\ (1_Q\Vdash\theta),
\]
where $\mathcal C_\sigma(V')$ is defined from the same first $\sigma$
entries of the list, with the wellorder of $V'$ at that stage as the
given wellorder. Tuples $\langle\SEP,c,a,e_0,e_1,x,i\rangle$ such that
$(c,a,e_0,e_1)$ occurs in no requirement of the list are unrestricted
at that rank. In particular, adding a requirement does not change the
test for an earlier requirement. Every amoeba forcing of $N_a$ with
$a$ in the current model is allowed.

Here and below the ground models \(V\) are extensions of \(W\)
by coding iterations of length less than \(\kappa\), possibly
restricted below a condition, and the forcings over \(V\) are tails
of length less than \(\kappa\). All names, including the names for
wellorders and lists, are translated along the specified complete
embeddings. It suffices to restrict below conditions of the dense
suborder.

\paragraph{The recursion on requirements.}
Rule~2 for a requirement with index \(\sigma\) is evaluated in the
extension \(V'\) by the preceding stages, using \(\mathcal L_\sigma\)
and the wellorder of \(V'\) at that stage. A forcing obeys the rule
if every such generic extension obeys it whenever it applies. Since
\(\sigma<\eta\) for each requirement in \(\mathcal L_\eta\), the
classes \(\mathcal C_\eta\) are defined by recursion on \(\eta\),
simultaneously in all the models under consideration. As in
\cite[Lemma~2.19]{piunif}, the recursion can be carried out in a
ground model.

Fix a model \(V\). Up to the specified presentations, the tails over
\(V\), their restrictions below conditions, and the names for
further tails form a set, as in \cite[Subsection~2.5]{piunif}. By
Lemma~\ref{lem:presentation}(4) their dense suborders can be carried
by subsets of \(\kappa\). Names for reals, for wellorders of the
reals and for lists of requirements are then coded by families of
antichains of these suborders. A \emph{candidate name} over such a
tail \(P\) is a name in this set for an ordered iteration over the
extension by \(P\). For a candidate whose length is given by a name
we work below a condition deciding the length.

\begin{lemma}\label{lem:recursion}
Let \(P\) be a tail over \(V\) in this set, together with names for
a wellorder of the reals and for a list of requirements in its
extension. By recursion on \(\eta\) there are sets
\(A_\eta(P,\dot Q)\) and \(\mathcal S_\eta(P,\dot\theta)\) of
conditions of \(P\) in \(V\), for candidate names \(\dot Q\) and
\(\Sigma^1_{n+3}\) sentences \(\theta\) with names for real
parameters, such that for every \(P\)-generic \(G\) over \(V\)
\begin{align}
 \dot Q^G\in\mathcal C_\eta(V[G])
 &\quad\Longleftrightarrow\quad
 G\cap A_\eta(P,\dot Q)\ne\varnothing,
 \tag{A}\label{eq:allowability-filter}\\
 \Poss_\eta^{V[G]}(\theta^G)
 &\quad\Longleftrightarrow\quad
 G\cap\mathcal S_\eta(P,\dot\theta)\ne\varnothing .
 \tag{P}\label{eq:possibility-filter}
\end{align}
\end{lemma}

\begin{proof}
We follow \cite[Lemma~2.19]{piunif}. Induct on \(\eta\),
simultaneously over the presentations, and interpret all
hypothetical generic filters by the forcing theorem. Suppose the
sets are defined below \(\eta\) and satisfy
\eqref{eq:allowability-filter} and \eqref{eq:possibility-filter}. A
candidate \(Q\) over \(V[G]\) belongs to \(\mathcal C_\eta(V[G])\)
if and only if its wellorders extend one another and satisfy the
end-extension requirements in \(\mathcal L_\eta\) and, for every
hypothetical \(V[G]\)-generic \(H\subseteq Q\) and every coding
stage \(i\), the tuple coded at stage \(i\) obeys rules~1 and~2 in
\(V[G][H_i]\). Let \(p\in G\) decide the length of \(\dot Q\). By
Lemma~\ref{lem:presentation}, \((P\restriction p)*\dot Q_i\) is a
tail over \(V\) with a presentation in our set, and \(V[G][H_i]\) is
a generic extension of \(V\) by this forcing. A requirement in
\(\mathcal L_\eta\) has an index \(\sigma<\eta\), and the two tests
in rule~2 are expressed by \eqref{eq:possibility-filter} at
\(\sigma\) for this forcing. Thus the induction hypothesis and the
forcing theorem express every clause by conditions and names in
\(V\). Let \(A_\eta(P,\dot Q)\) be the set of conditions of \(P\)
forcing all the clauses. The truth lemma gives
\eqref{eq:allowability-filter}. Put
\[
 \mathcal S_\eta(P,\dot\theta)=
 \{p\in P:\exists\dot R\
     [p\in A_\eta(P,\dot R)\ \land\
      p\Vdash_P (1_{\dot R}\Vdash\dot\theta)]\},
\]
where \(\dot R\) ranges over the candidate names. A condition in
\(G\cap\mathcal S_\eta(P,\dot\theta)\) supplies a witness for
\(\Poss_\eta^{V[G]}(\theta^G)\). Conversely, name a witness and take
a common strengthening in \(G\) of a condition in \(A_\eta\) for
this name and a condition forcing that it forces \(\theta\). See
\cite[Sections~4--5]{sepred} and \cite{lsepwo} for the corresponding
recursions over \(L\).
\end{proof}

\begin{lemma}\label{lem:amalgam}
Let $P\in\mathcal C_\eta(V)$, let $G$ be $P$-generic over $V$, and
let $G_i=G\cap P_i$. For every stage $i$ and every $p$ in the dense
suborder of $P$,
\[
 P_i,\ P\restriction p\in\mathcal C_\eta(V),
 \qquad P/G_i\in\mathcal C_\eta(V[G_i]).
\]
If $p$ decides the length of $\dot Q$, then
\[
 p\Vdash_P\dot Q\in\mathcal C_\eta(V[\dot G])
 \quad\Longrightarrow\quad
 (P\restriction p)*\dot Q\in\mathcal C_\eta(V).
\]
In particular, $\mathcal C_\eta$ is closed under such finite
concatenations.

For $P,Q\in\mathcal C_\eta(V)$, the intermediate models of $P\times Q$
admit coherent wellorders extending the orders of both factors. With
these wellorders the product, presented with either factor first as
in Lemma~\ref{lem:presentation}(3), belongs to $\mathcal C_\eta(V)$.
For every $\Sigma^1_{n+3}$ sentence
$\theta$ with real parameters in $V$,
\[
                 \Poss_\eta^{V}(\theta)
          \quad\Longleftrightarrow\quad
                 \Poss_\eta^{V[G]}(\theta),
\]
using the same list and the final wellorder of $P$ in $V[G]$.
\end{lemma}

\begin{proof}
\emph{Operations on iterations.}
By Lemma~\ref{lem:presentation}(1)--(3), initial segments, quotients
by initial segments, restrictions below conditions of the dense
suborder, concatenations \((P\restriction p)*\dot Q\) with \(p\)
deciding the length of \(\dot Q\), and products on disjoint Cohen
index sets are again tails over the respective models. The stage
parameters and the names for wellorders and lists are translated
along the specified embeddings and have the same evaluations. A
bijection in \(M\) between Cohen index sets which fixes the index
sets used to obtain \(V\) translates names by a map \(\pi\) with
\((\pi\dot a)^{\pi G}=\dot a^G\). Induction on \(\sigma\) gives
invariance of \(\Poss_\sigma\) under this translation.

For \(P,Q\) over \(V\), write \(d,e\) for their lengths and use the
product presentations of Lemma~\ref{lem:presentation}(3). For
restrictions \(G_i,H_j\) of product generics, put
\[
 \begin{gathered}
 V_{ij}=V[G_i,H_j],\\
 (P_i\times Q_j)/G_i\simeq Q_j^{(i)}\text{ over }V_{i0},
 \qquad
 (P_i\times Q_j)/H_j\simeq P_i^{(j)}\text{ over }V_{0j}.
 \end{gathered}
\]
The respective intermediate models are $V_{i\ell}$ for $\ell\leq j$
and $V_{kj}$ for $k\leq i$.

\smallskip
\noindent\emph{Coherent wellorders.}
Write $R_{ij}=\R^{V_{ij}}$.
For all $i,i'\leq d$ and $j,j'\leq e$,
\[
 R_{ij}\cap R_{i'j'}=R_{\min(i,i'),\min(j,j')}.
 \tag{*}
\]
For incomparable pairs, factor over the model on the right. The two
remaining quotients are mutually generic. If $(p,q)$ in their product
forces $\dot x=\dot y$ for names from the first and second quotient,
respectively, fix $m<\omega$ and $q'\leq q$ with
$q'\Vdash\dot y(m)=k$. Then
\[
 (p,q')\Vdash\dot x(m)=k
 \quad\Longrightarrow\quad p\Vdash\dot x(m)=k .
\]
Thus $p$ decides $\dot x$, and its value belongs to the common model.
For comparable pairs, $(*)$ follows from inclusion.

Enumerate $R_{00}$ in its given order as $\langle a_\xi:\xi<\lambda\rangle$.
In $V_{ij}$ put $R^0=R_{i0}$ and $R^1=R_{0j}$, with factor orders
$\prec_0,\prec_1$. For $x\in R^t\setminus R_{00}$ let
\[
       \operatorname{cut}_t(x)=\min(\{\xi<\lambda:x\prec_t a_\xi\}\cup\{\lambda\}).
\]
For $x\in R_{ij}\setminus(R^0\cup R^1)$ let
\[
 k(x)=\min\{k\leq i:x\in R_{kj}\},\qquad
 \ell(x)=\min\{\ell\leq j:x\in R_{i\ell}\}.
\]
By $(*)$, $(k(x),\ell(x))$ is the least index pair containing $x$.
Fix in $V$ a wellordered list of real names for each $P_k\times Q_\ell$.
Let $m(x)$ be the first index of a name evaluating to $x$ at
$(k(x),\ell(x))$. Order $R_{ij}$ by the lexicographic order on
\[
 K_{ij}(x)=
 \begin{cases}
 (0,\xi,2,0), & x=a_\xi,\\
 (0,\operatorname{cut}_t(x),t,\operatorname{rk}_{\prec_t}(x)),
                 & x\in R^t\setminus R_{00},\\
 (1,k(x),\ell(x),m(x)), & x\notin R^0\cup R^1 .
 \end{cases}
\]
Here $\operatorname{rk}_{\prec_t}(x)$ is the order type of the
predecessors of $x$ in $\prec_t$. The keys are injective. On either
factor the cuts and the comparisons inside each cut recover its
order. On passing to a larger intermediate model, $(*)$ preserves
the three cases and the keys in the last case. In the second case
only the ranks can change, and
\[
 \operatorname{rk}_{\prec_t}(x)<\operatorname{rk}_{\prec_t}(y)
 \quad\Longleftrightarrow\quad x\prec_t y
\]
is unchanged for old reals. Consequently
\[
        \prec_{ij}\restriction R_{i'j'}=\prec_{i'j'}
                      \qquad(i'\leq i,\ j'\leq j).
\]
Each segment of $R_{00}$ subject to an end-extension requirement
remains initial on both
factors, hence in $\prec_{ij}$ by the cuts and the leading coordinate
of the keys. The definitions use only names and generics of $V_{ij}$,
so they give order names at every intermediate model.

\smallskip
\noindent\emph{Induction on $\eta$.}
This is the induction of \cite[Lemma~3.8]{reduction}, with the
separation tests in place of the rules for reduction.
Induct simultaneously over all ground models, requirement lists and
orders under consideration. At rank $\eta$, first prove the closure
properties and, for \emph{every} coherent system extending the factor
orders and respecting the end-extension requirements, the implication
\[
 P,Q\in\mathcal C_\eta(V)
 \quad\Longrightarrow\quad
 \begin{cases}
 Q_j^{(i)}\in\mathcal C_\eta(V_{i0}),\\
 P_i^{(j)}\in\mathcal C_\eta(V_{0j})
 \end{cases}
 \qquad(i\leq d,\ j\leq e).
 \tag{$\dagger_\eta$}
\]
The preceding operations preserve the evaluated stage data, orders
and tests for initial segments, quotients and restrictions. For
concatenation the wellorder at the start of $\dot Q$ is the final wellorder
of $P$, giving closure.

Stage $i$ of $P$ in $P^{(j)}$ uses the same tuple and real
parameters as in $P$. The $\WO$ rule is preserved by
$\prec_{ij}\restriction R_{i0}=\prec_{i0}$. For a $\SEP$ rule
referring to $\sigma<\eta$, put $\theta_t=\varphi_{e_t}(x,a)$
for $t<2$. Apply $(\dagger_\sigma)$ to $P_i,Q_j$ with the
restriction of the given system. The inductive invariance at
rank $\sigma$ now gives
\[
 \begin{gathered}
 Q_j^{(i)}\in\mathcal C_\sigma(V_{i0}),\\
 \Poss_\sigma^{V_{i0}}(\theta_t)
       \quad\Longleftrightarrow\quad
 \Poss_\sigma^{V_{ij}}(\theta_t)
                  \qquad(t<2).
 \end{gathered}
\]
The final wellorder in this test is $\prec_{ij}$. Thus both tests
and the prescribed bit agree. For stages of $Q$, use $P_i^{(j)}$
and interchange the two factors. Rank $0$ has no separation tests,
and at limit ranks every separation rule refers to a lower rank.
This proves $(\dagger_\eta)$, and concatenation gives the two
allowable orders of iteration.

Finally, let $G$ be $P$-generic. If $Q$ witnesses
$\Poss_\eta^V(\theta)$, Lemma~\ref{lem:absoluteness} and the
product presentation give
\[
 1_Q\Vdash\theta
    \quad\Longrightarrow\quad P\times Q\Vdash\theta
    \quad\Longrightarrow\quad V[G]\models 1_{Q_e^{(d)}}\Vdash\theta .
\]
Together with $(\dagger_\eta)$ this gives
$\Poss_\eta^{V[G]}(\theta)$. Conversely choose a name $\dot Q$ for
a witness in $V[G]$ and a condition $p\in G$ of the dense suborder
deciding its length with
\[
 \begin{aligned}
 p&\Vdash_P \dot Q\in\mathcal C_\eta(V[\dot G])
                       \ \land\ (1_{\dot Q}\Vdash\theta),\\
 (P\restriction p)*\dot Q&\in\mathcal C_\eta(V),\\
 (P\restriction p)*\dot Q&\Vdash\theta.
 \end{aligned}
\]
The last two lines follow by closure and the forcing theorem,
respectively. For the applications of
Lemma~\ref{lem:absoluteness}, precede each product or iteration by
\(\Bb\) and the coding iteration producing $V$ from $W$. By
Lemma~\ref{lem:presentation}(4), the extensions are small
\(M_n\)-generic extensions.
\end{proof}

\section{The final iteration}\label{sec:iteration}

We construct a coding iteration
\(\langle P_\beta:\beta\leq\kappa\rangle\) over \(W\).
For \(G_\beta\subseteq P_\beta\) generic, write
\(V_\beta=W[G_\beta]\). We specify the choices in these models and
use the forcing theorem to name them in \(W\). The iteration is
divided into consecutive \emph{segments}. For a given generic,
write \(\alpha_\xi\) for the first stage of segment \(\xi\).
Each segment has a length \(\ell_\xi\) with
\(0<\ell_\xi<\kappa\), so segment \(\xi\) ends at
\(\alpha_{\xi+1}=\alpha_\xi+\ell_\xi\).

\paragraph{Bookkeeping.}
Once \(P_\beta\) has been defined, use
Lemma~\ref{lem:presentation}(4) and the maximum principle to fix a name
\(\dot f_\beta:\kappa\to\R\) which enumerates the reals of
\(V_\beta\). The tasks are
\[
 (a,e_0,e_1),\qquad
 (a,e_0,e_1,x),\qquad
 (\WO,x,y),\qquad
 (\mathcal I,a)\quad(\mathcal I\in\{\mathcal N,\mathcal M\}),
\]
where each displayed real is specified by an index \((\beta,\nu)\)
for \(\dot f_\beta(\nu)\). Fix in \(W\) a schedule of length
\(\kappa\) in which each task occurs cofinally often. At segment
\(\xi\) use its scheduled task if all indices \((\beta,\nu)\) in it
satisfy \(\beta\leq\alpha_\xi\), and use a trivial stage otherwise.
Since \(\alpha_\xi\geq\xi\), each task occurs cofinally often after
its names are available. In particular the schedule covers reals
added at limit stages.

\paragraph{Orders and the list.}
Choose a wellorder of the reals of \(W\). At the first stage
\(\alpha_\xi\) of each segment, add to the list the end-extension
requirement for the whole wellordered set \(\R^{V_{\alpha_\xi}}\).
The wellorder at the end of a segment is the final wellorder of the
forcing used in it. If \(\xi\) is a nonzero limit of segment indices, set
\[
 \alpha_\xi=\sup_{\zeta<\xi}\alpha_\zeta,\qquad
 R_{<\xi}=\bigcup_{\zeta<\xi}\R^{V_{\alpha_\zeta}},\qquad
 \prec_{<\xi}=\bigcup_{\zeta<\xi}\prec_{\alpha_\zeta}.
\]
The boundary orders end-extend one another, because each segment
preserves the whole ordered set at its first stage as an initial
segment. Thus \(\prec_{<\xi}\) wellorders \(R_{<\xi}\). At
\(\alpha_\xi\), append a wellorder of
\(\R^{V_{\alpha_\xi}}\setminus R_{<\xi}\). The list at a limit of
segments is the union of the earlier lists, before adding the
new end-extension requirement. At the intermediate stages of a witness
segment, including limit stages, use the wellorders that come with
the witness. If segment \(\xi+1\) begins at a limit ordinal, its first
wellorder is still the final wellorder of segment \(\xi\).

\paragraph{The segment at a task.}
Work in \(V_{\alpha_\xi}\) with the list after the new end-extension
requirement has been entered.
\begin{enumerate}
\item For \((a,e_0,e_1)\), do nothing if the triple was processed
earlier. Otherwise let \(\mathcal L_\sigma\) be the current list and ask
whether
\[
 \Poss_\sigma^{V_{\alpha_\xi}}\!
    \left(\exists x\,
      [\varphi_{e_0}(x,a)\land\varphi_{e_1}(x,a)]\right)
\]
holds. If it holds, use a witness in
\(\mathcal C_\sigma(V_{\alpha_\xi})\) for the next segment.
A witness with no stages may be followed by a trivial stage.
If the test fails, use a Cohen stage. After this stage append
\((c,a,e_0,e_1,\sigma)\) to the list, where \(c\) is the Cohen
real. It is new over \(V_{\alpha_\xi}\). In either case mark the
triple as processed.
\item For \((a,e_0,e_1,x)\) such that a requirement
\((c,a,e_0,e_1,\sigma)\) is on the current list, code
\[
 \langle\SEP,c,a,e_0,e_1,x,i\rangle ,
\]
with \(i\) as prescribed by rule~2 of Section~\ref{sec:allowable}.
\item For \((\WO,x,y)\) with \(x\ne y\), code
\(\langle\WO,x,y\rangle\) if \(x\prec_{\alpha_\xi}y\), and
\(\langle\WO,y,x\rangle\) otherwise.
\item For \((\mathcal I,a)\), use the amoeba forcing for
\(\mathcal I\) computed in \(L[T,a]\).
\end{enumerate}
A task requiring no action is assigned a trivial stage.
For the segments consisting of a single stage, the final wellorder
is the current one followed by a wellorder of the new reals; their
coding stages obey the current rules by construction. The witness
segments come with names for their wellorders and obey the rules at
each of their stages.

\paragraph{The iteration as a single forcing.}
The construction is carried out in \(W\) as in the proof of
\cite[Theorem~3.1]{piunif}. At a segment start \(\alpha\), the task
is evaluated in \(V_\alpha=W[G_\alpha]\). Take a maximal antichain
of conditions of \(P_\alpha\) which force that a segment starts at
\(\alpha\) and decide which clause applies. Where a witness is
required, they also choose a candidate name for it and decide its
length. Such conditions are dense by
\eqref{eq:possibility-filter} of Lemma~\ref{lem:recursion} and the
forcing theorem. On the other
pieces use the prescribed single stage, with names for its
wellorders given by the forcing theorem. At the \(k\)-th stage within
a segment, mix the names for the \(k\)-th stage parameters and
wellorders along the antichain, translated to stage \(\alpha+k\) by
Lemma~\ref{lem:presentation}. At limit stages take the limits of the
initial segments and use the rule for wellorders given above. Thus
the stage recursion takes place on the fixed ordinal \(\kappa\),
even when the length of a segment depends on the generic at its
start. All \(\alpha_\xi\) for \(\xi<\kappa\) are below \(\kappa\) by
regularity, and their supremum is \(\kappa\) because each segment
has positive length.

\paragraph{Allowability of tails.}
Suppose \(\beta\leq\gamma<\kappa\) and the first \(\sigma\)
entries of the list were present at \(V_\beta\). The quotient forcing
for the interval \([\beta,\gamma)\) keeps the given wellorders at all
of its stages. Each end-extension requirement of
\(\mathcal L_\sigma\) is respected through the remainder of the
segment containing \(\beta\) and through every subsequent
segment. At a limit of segments the required segment remains an
initial segment of the union order. A coding stage in
\([\beta,\gamma)\) obeys the \(\WO\) rule and all requirements
on the list in force at that stage. In particular it obeys
the requirements in \(\mathcal L_\sigma\), because the tests for
those requirements still use their original indices.
Hence the tail, with these wellorders, is in
\(\mathcal C_\sigma(V_\beta)\). This verification applies
also when the interval meets infinitely many segments.

Let \(V^*=W[G_\kappa]\). By Lemma~\ref{lem:presentation}(4),
\[
 \R^{V^*}=\bigcup_{\beta<\kappa}\R^{V_\beta},
\]
and the models \(V_\beta\), \(\beta\leq\kappa\), are small
\(M_n\)-generic extensions satisfying \(\CH\).

\section{Verification}\label{sec:verification}

\begin{proposition}\label{prop:sep}
In \(V^*\), for \(a\in\R\) and \(e_0,e_1<\omega\),
\[
 A_{e_0}(a)\cap A_{e_1}(a)=\varnothing
 \ \Longrightarrow\
 \exists X\in\Db^1_{n+3}\
 \bigl(A_{e_0}(a)\subseteq X\subseteq\R\setminus A_{e_1}(a)\bigr).
\]
\end{proposition}

\begin{proof}
Write \(\theta_i(x)=\varphi_{e_i}(x,a)\) for \(i<2\) and
\(\theta=\exists x\,(\theta_0(x)\land\theta_1(x))\).
At the first task for \((a,e_0,e_1)\), with current list \(\mathcal L_\sigma\)
at stage \(\alpha\),
\[
 \Poss_\sigma^{V_\alpha}(\theta)
 \ \Longrightarrow\ V^*\models\theta
\]
by the choice of witness segments and Lemma~\ref{lem:absoluteness}.
Since \(V^*\models\neg\theta\), the next stage adds the Cohen real
\(c\notin V_\alpha\) and the requirement \((c,a,e_0,e_1,\sigma)\).
Thus no earlier stage codes a tuple containing \(c\).

For \(\alpha<\beta<\kappa\), Lemma~\ref{lem:amalgam} gives
\(\neg\Poss_\sigma^{V_\beta}(\theta)\). If both tests for a real
\(x\in V_\beta\) were positive, with witnesses \(Q_0,Q_1\), then
\[
 Q_i\Vdash\theta_i(x)\quad(i<2)
 \ \Longrightarrow\
 Q_0\times Q_1\Vdash\theta_0(x)\land\theta_1(x).
\]
This uses Lemma~\ref{lem:absoluteness}. The product is
\(\sigma\)-allowable by Lemma~\ref{lem:amalgam}, a contradiction.
Hence rule~2
assigns side \(1\) exactly when \(\Poss_\sigma^{V_\beta}(\theta_1(x))\)
holds. By invariance along \(\sigma\)-allowable tails this side,
denoted \(s(x)\), is independent of \(\beta\) once \(x\in V_\beta\).
Put \(u_i(x)=\langle\SEP,c,a,e_0,e_1,x,i\rangle\). The bookkeeping,
rule~2 and Lemma~\ref{lem:exact} give
\[
 \Phi_n(u_i(x))\ \Longleftrightarrow\ i=s(x),\qquad
 X_i=\{x:\Phi_n(u_i(x))\},\qquad
 X_0\mathbin{\dot\cup}X_1=\R^{V^*}.
\]
If \(x\in A_{e_i}(a)\), choose \(\beta>\alpha\) such that \(V_\beta\) contains \(x\), \(a\)
and a real witness for \(\theta_i(x)\). Then
\[
 V_\beta\models\theta_i(x)
 \ \Longrightarrow\
 \Poss_\sigma^{V_\beta}(\theta_i(x))
 \ \Longrightarrow\ s(x)=i
 \ \Longrightarrow\ x\in X_i.
\]
The first assertion follows from Lemma~\ref{lem:absoluteness} and
the first implication uses the trivial forcing. Since
\(X_i\in\Sigma^1_{n+3}(c,a)\) and \(X_0=\R^{V^*}\setminus X_1\),
\(X_0\) is a \(\Db^1_{n+3}\) separator.
\end{proof}

\begin{proposition}\label{prop:wo}
In \(V^*\) there is a lightface \(\Delta^1_{n+3}\) relation
\(\prec\,\subseteq(\R^{V^*})^2\) wellordering \(\R^{V^*}\).
\end{proposition}

\begin{proof}
At the segment boundaries the orders end-extend one another, and
\[
 \prec=\bigcup_{\xi<\kappa}\prec_{\alpha_\xi},\qquad
 \R^{V^*}=\bigcup_{\xi<\kappa}\R^{V_{\alpha_\xi}}.
\]
Thus \(\prec\) wellorders \(\R^{V^*}\).
Rule~1 and coherence give
\(\langle\WO,x,y\rangle\text{ coded}\ \Longrightarrow\ x\prec y\),
including codes in witness segments. Bookkeeping gives the converse.
By Lemma~\ref{lem:exact},
\[
 x\prec y
 \quad\Longleftrightarrow\quad
 x\ne y\ \land\ \Phi_n(\langle\WO,x,y\rangle)
 \quad\Longleftrightarrow\quad
 x\ne y\ \land\ \neg\Phi_n(\langle\WO,y,x\rangle).
\]
The two parameter-free definitions are \(\Sigma^1_{n+3}\) and
\(\Pi^1_{n+3}\), respectively.
\end{proof}

\begin{proposition}\label{prop:reg}
In \(V^*\), every \(\Sb^1_{n+2}\) set is Lebesgue measurable and has
the Baire property.
\end{proposition}

\begin{proof}
Fix \(e<\omega\), \(a\in\R^{V^*}\), and put
\[
 A=\{z\in\R^{V^*}:V^*\models U(e,a,z)\},\qquad N_a=L[T,a].
\]
For \(\mathcal I\in\{\mathcal N,\mathcal M\}\), the scheduled
amoeba stage over \(N_a\) gives a Borel \(Z_{\mathcal I}\) with
\[
 Z_{\mathcal I}\in\mathcal I,\qquad
 \bigcup\{B\in\mathcal I:B\text{ has a Borel code in }N_a\}
       \subseteq Z_{\mathcal I}
 \tag{2}\label{eq:amoeba-cover}
\]
in \(V^*\), by Subsection~\ref{subsec:amoeba}. In \(N_a\) let
\(\mathbb U_{\mathcal N}\) and \(\mathbb U_{\mathcal M}\) be the
random and Cohen Boolean algebras, respectively. Choose
\(B_{\mathcal I}\) with Borel code in \(N_a\) such that
\[
 [B_{\mathcal I}]_{\mathcal I}
   =\bigl\|[T_{e,a,\dot z}]\ne\varnothing\bigr\|_{\mathbb U_{\mathcal I}}^{N_a},
\]
where \(\dot z\) names the generic real. By
\eqref{eq:amoeba-cover}, every \(z\notin Z_{\mathcal I}\) is
\(\mathbb U_{\mathcal I}\)-generic over \(N_a\). For such \(z\),
\[
 z\in B_{\mathcal I}\ \Longleftrightarrow\
 [T_{e,a,z}]^{N_a[z]}\ne\varnothing\ \Longleftrightarrow\
 [T_{e,a,z}]^{V^*}\ne\varnothing\ \Longleftrightarrow\ z\in A.
\]
The middle equivalence uses a branch or a rank function in \(N_a[z]\).
The last is Lemma~\ref{lem:tree}. Consequently
\[
 A\triangle B_{\mathcal I}\subseteq Z_{\mathcal I}\in\mathcal I
 \qquad(\mathcal I=\mathcal N,\mathcal M).
\]
Thus \(A\) is Lebesgue measurable and has the Baire property.
\end{proof}

\begin{proposition}\label{prop:nondeterminacy}
In \(V^*\), \(\Pb^1_{n+1}\)-determinacy fails. If \(n\) is odd,
then \(\Db^1_{n+1}\)-determinacy also fails.
\end{proposition}

\begin{proof}
Suppose \(V^*\models\Pb^1_{n+1}\text{-determinacy}\).
By \cite[Theorem~2.1]{msw}, \(M_n^\#\) exists in \(V^*\)
and is \(\omega_1\)-iterable there. For \(x\in\R^{M_n}\),
\cite[Lemma~4.9]{piunif} gives a premouse \(P\) with a code in
\(\mathcal I_n\) and \(x\in P\). It is a sound \(n\)-small initial
segment of \(M_n\) with \(\rho_\omega(P)=\omega\), and it has a code
in \(\mathcal I_n\) also in the small \(M_n\)-generic extension
\(V^*\).
Sound iterable premice projecting to \(\omega\) are linearly
ordered by the initial segment relation
\cite[Section~2]{steelwellorder}. Since \(P\) is \(n\)-small and
\(M_n^\#\) is not, \(P\lhd M_n^\#\), and so \(x\in M_n^\#\).
Consequently
\[
                    \R^{M_n}\subseteq M_n^\#.
\]
Since \(M_n^\#\) is sound and \(\rho_\omega(M_n^\#)=\omega\), it is
countable in \(V^*\), whereas
\(\kappa=\omega_1^{M_n}=\omega_1^{V^*}\). This contradicts
\(M_n\models\CH\).

For odd \(n\), \(\Db^1_{n+1}\)-determinacy implies
\(\Sb^1_{n+3}\)-uniformization and hence
\(\Sb^1_{n+3}\)-reduction. This contradicts
Proposition~\ref{prop:sep}, by
\cite[Theorem~6C.5 and Exercise~4B.12]{moschovakis}.
\end{proof}

\begin{proof}[Proof of Theorem~\ref{thm:main}]
Use the extension \(V^*=W[G_\kappa]\) by \(\Bb*\dot P_\kappa\).
Subsection~\ref{sec:coding-forcings} and
Lemma~\ref{lem:presentation}(4) give cardinal preservation and
\(V^*\models\CH\).
Propositions~\ref{prop:sep} and~\ref{prop:wo} give separation and
the lightface wellorder. By Lemma~\ref{lem:good-wellorder},
\(\Sb^1_{n+3}\)-separation excludes every good \(\Db^1_{n+3}\)
wellorder in \(V^*\). Propositions~\ref{prop:reg}
and~\ref{prop:nondeterminacy} give regularity and the stated
failures of determinacy.
\end{proof}

\section*{Use of AI}

Theorem~\ref{thm:main} was obtained in May 2026,
without the use of AI. It was a step towards a stronger statement, namely that $\Pi^1_3$ uniformization and a $\Delta^1_3$ wellorder of the reals are consistent,
claimed in the earlier version of this paper, which turned out to be
false and which the present version replaces.

Theorem~\ref{thm:reduction-wellorder} is joint work with the AI
systems Claude Fable~5.2 (Anthropic) and \mbox{GPT-6.2}~Astra (OpenAI).
It came after many failed attempts to repair the stronger statement.
These attempts eventually left almost no room for a forcing that gives a
\(\Db^1_k\) wellorder together with \(\Pb^1_k\)-reduction, which led
me to suspect that the two are inconsistent. The current proof was found by \mbox{GPT-6.2}~Astra.

AI systems were also used for checking proofs and for revising the
exposition of the present version. I take full responsibility for the content of the
paper.

\end{document}